%% file: doc.tex
\documentclass[a4paper,11pt,reqno]{customart}

\usepackage[utf8x]{inputenc}
\usepackage[margin=1in]{geometry}
\usepackage{verbatim}
\usepackage{color}
\usepackage[table]{xcolor}
\usepackage{listings}
\usepackage{amssymb,amsfonts,amsthm,amsmath}
\usepackage{url}
\usepackage{enumerate}
\usepackage{todonotes}
\usepackage[all]{xy}
\usepackage{multirow}
\usepackage{tikz}
\usepackage[underline=false]{pgf-umlsd}
\usepackage{stmaryrd}
\usepackage{footnote}
\makesavenoteenv{tabular}
\usepackage{hyperref}

\makeatletter
\newcommand\footnoteref[1]{\protected@xdef\@thefnmark{\ref{#1}}\@footnotemark}
\makeatother

\hypersetup{
	pdfborder={0 0 0}
}

\definecolor{grey}{rgb}{0.95,0.95,0.95}
\definecolor{green}{rgb}{0.2,0.6,0.4}

\renewcommand{\O}{\textbf{0}}

\newcommand{\Psf}{\mathsf{P}}
\newcommand{\Qsf}{\mathsf{Q}}

\newcommand{\biimp}{\leftrightarrow}

\newcommand{\Ccal}{\mathcal{C}}

\newcommand{\Ical}{\mathcal{I}}

\newcommand{\Mcal}{\mathcal{M}}
\newcommand{\Ncal}{\mathcal{N}}

\newcommand{\Qcal}{\mathcal{Q}}
\newcommand{\Rcal}{\mathcal{R}}
\newcommand{\Scal}{\mathcal{S}}

\newcommand{\uh}{{\upharpoonright}}

\renewcommand{\setminus}{\smallsetminus}

\newcommand{\tuple}[1]{\left\langle #1 \right\rangle}

\newcommand{\s}[1]{\ensuremath{\sf{#1}}}

\DeclareMathOperator{\rca}{\s{RCA}_0}
\DeclareMathOperator{\aca}{\s{ACA}_0}
\DeclareMathOperator{\wkl}{\s{WKL}_0}

\DeclareMathOperator{\rwkl}{\s{RWKL}}

\DeclareMathOperator{\rt}{\s{RT}}
\DeclareMathOperator{\srt}{\s{SRT}}

\DeclareMathOperator{\cac}{\s{CAC}}
\DeclareMathOperator{\tto}{\s{TT}}

\DeclareMathOperator{\coh}{\s{COH}}

\usetikzlibrary{shapes,arrows}
\usetikzlibrary{decorations.markings}
\definecolor{lightblue}{HTML}{e6e6e6}
\definecolor{lightred}{HTML}{eca6a6}
\definecolor{lightgreen}{RGB}{164,244,140}

\newtheoremstyle{custom}% name of the style to be used
  {10pt}% measure of space to leave above the theorem. E.g.: 3pt
  {10pt}% measure of space to leave below the theorem. E.g.: 3pt
  {\normalfont}% name of font to use in the body of the theorem
  {}% measure of space to indent
  {\bfseries}% name of head font
  {}% punctuation between head and body
  { }% space after theorem head; " " = normal interword space
  {}% Manually specify

\theoremstyle{custom}

\usepackage{xcolor}	
\usepackage{soul}

\newtheorem{theorem}{Theorem}[section]
\newtheorem{lemma}[theorem]{Lemma}
\newtheorem{definition}[theorem]{Definition}

\newtheorem{question}[theorem]{Question}

\newtheorem{corollary}[theorem]{Corollary}

\begin{document}

\title[The strength of the tree theorem for pairs]{The strength of the tree theorem for pairs\\
	in reverse mathematics}
\author{
  Ludovic Patey
}

\begin{abstract}
No natural principle is currently known to be strictly between the arithmetic comprehension axiom
($\aca$) and Ramsey's theorem for pairs ($\rt^2_2$) in reverse mathematics.
The tree theorem for pairs ($\tto^2_2$) is however a good candidate.
The tree theorem states that for every finite coloring over tuples of comparable nodes in
the full binary tree, there is a monochromatic subtree isomorphic to the full tree.
The principle $\tto^2_2$ is known to lie between $\aca$ and~$\rt^2_2$ over~$\rca$, but its exact strength remains open.
In this paper, we prove that $\rt^2_2$ together with weak König's lemma ($\wkl$)
does not imply $\tto^2_2$, thereby answering a question of Mont\'alban. 
This separation is a case in point of the method of Lerman, Solomon and Towsner
for designing a computability-theoretic property which discriminates between two statements in reverse
mathematics. We therefore put the emphasis on the different steps leading to this separation
in order to serve as a tutorial for separating principles in reverse mathematics.
\end{abstract}

\maketitle

\section{Introduction}
\input{parts/part0-introduction}

\section{Partitioning trees and strong reducibility}\label{sect:partitions-reducibility}
\input{parts/part1-partitions}

\section{The fairness property}
\input{parts/part2-fairness}

\section{Separating principles in reverse mathematics}\label{sect:separating-principles}
\input{parts/part3-separations}

\section{Questions}
\input{parts/part4-questions}

\vspace{0.5cm}

\noindent \textbf{Acknowledgements}. The author is thankful to his PhD advisor Laurent Bienvenu
for interesting comments and discussions.
The author is funded by the John Templeton Foundation (`Structure and Randomness in the Theory of Computation' project). The opinions expressed in this publication are those of the author(s) and do not necessarily reflect the views of the John Templeton Foundation.

\end{document}

%% file: parts/part0-introduction.tex
``Every sufficiently large collection of objects has an arbitrarily
large sub-collection whose objects satisfy some structural properties''.
This general statement reflects the main idea of Ramsey's theory.
This theory has connections with many areas of mathematics,
namely, combinatorics, model theory or set theory.
One of the most well-known statements is Ramsey's theorem, stating
that for every coloring of the $k$-tuples of integers in finitely many colors,
there is an infinite monochromatic subset.
In this paper, we are interested in the tree theorem for pairs,
a strengthening of Ramsey's theorem for pairs stating that 
for every finite coloring over pairs of comparable nodes in the full binary tree,
there is a monochromatic subtree isomorphic to the full tree.
Our main theorem states that the tree theorem for pairs is strictly stronger 
than Ramsey's theorem for pairs in the sense of \emph{reverse mathematics}.

Reverse mathematics is a mathematical program whose goal is
to classify theorems in terms of their provability strength.
It uses the framework of subsystems of second-order arithmetic,
with the base theory $\rca$ (recursive comprehension axiom).
$\rca$ is composed of the basic first-order Peano axioms,
together with $\Delta^0_1$-comprehension and $\Sigma^0_1$-induction schemes.
$\rca$ is usually thought of as capturing \emph{computational mathematics}.
This program led to two important observations:
First, most ``ordinary'' (i.e.\ non set-theoreric) theorems require only very weak set existence axioms.
Second, many of those theorems are actually \emph{equivalent}
	to one of five main subsystems over $\rca$, known as the ``Big Five''.

Ramsey's theory, among others, provides a large class of theorems escaping this phenomenon.
Indeed, consequences of Ramsey's theorem for pairs usually belong to their own subsystem
and their study is still an active research subject within reverse mathematics.
This article focuses on Ramseyan principles below~$\aca$, the arithmetic comprehension axiom. 
See~\cite{Hirschfeldt2014Slicing} for a good introduction
to reverse mathematics.

\subsection{Ramsey's theorem}

The strength of Ramsey-type statements is notoriously hard to tackle
in the setting of reverse mathematics. The separation of Ramsey's theorem for pairs ($\rt^2_2$)
from the arithmetic comprehension axiom ($\aca$) was a long-standing open
problem, until Seetapun solved it~\cite{Seetapun1995strength} using the notion of cone avoidance.

\begin{definition}[Ramsey's theorem]
A subset~$H$ of~$\omega$ is~\emph{homogeneous} for a coloring~$f : [\omega]^k \to n$ (or \emph{$f$-homogeneous}) 
if each $k$-tuples over~$H$ are given the same color by~$f$. 
$\rt^k_n$ is the statement ``Every coloring $f : [\omega]^k \to n$ has an infinite $f$-homogeneous set''.
\end{definition}

Simpson~\cite[Theorem III.7.6]{Simpson2009Subsystems} proved that whenever~$k \geq 3$ and~$n \geq 2$,
$\rca \vdash \rt^k_n \biimp \aca$. Ramsey's theorem for pairs is probably the most famous 
example of statement escaping the Big Five. Seetapun~\cite{Seetapun1995strength} proved that~$\rt^2_2$ 
is strictly weaker than~$\aca$ over~$\rca$. Because of the complexity of the related separations, 
$\rt^2_2$ received a particular attention from the reverse mathematics community~\cite{Cholak2001strength,
Seetapun1995strength,Jockusch1972Ramseys}. 
Cholak, Jockusch and Slaman~\cite{Cholak2001strength} and Liu~\cite{Liu2012RT22} proved that 
$\rt^2_2$ is incomparable with weak K\"onig's lemma.
Dorais, Dzhafarov, Hirst, Mileti and Shafer~\cite{Dorais2012uniform},
Dzhafarov~\cite{Dzhafarov2014Cohesive}, Hirschfeldt and Jockusch~\cite{Hirschfeldtnotions}
and the author~\cite{Patey2015weakness} studied the computational strength of Ramsey's theorem
according to the number of colors, when fixing the number of applications of the principle.

\subsection{The tree theorem}

There is no natural principle currently known to be strictly between $\aca$
and~$\rt^2_2$. The tree theorem for pairs is however a good candidate.
The tree theorem is a strengthening of Ramsey's theorem in which we do not consider
colorings over tuples of integers, but colorings over tuples of nodes over a binary tree.
Ramsey's theorem can be recovered from the tree theorem by identifying
all nodes at every given level of the tree.

Given a set of binary strings $S \subseteq 2^{<\omega}$, we denote by $[S]^n$
the collection of \emph{linearly ordered} subsets of $S$ of size~$n$,
that is, $n$-sets of strings~$\{\sigma_0, \dots, \sigma_{n-1} \} \subseteq S$ 
such that~$\sigma_i \prec \sigma_{i+1}$ for each~$i < n-1$.

\begin{definition}[Tree theorem]
A subtree $S \subseteq 2^{<\omega}$ is \emph{order isomorphic} to $2^{<\omega}$ (written $S \cong 2^{<\omega}$) 
if there is a bijection $g : 2^{<\omega} \to S$ such that for all $\sigma, \tau \in 2^{<\omega}$,
$\sigma \preceq \tau$ if and only if $g(\sigma) \preceq g(\tau)$.
Given a coloring $f : [2^{<\omega}]^n \to k$, a tree $S$ is $f$-homogeneous if $S \cong 2^{<\omega}$
and $f \uh [S]^n$ is monochromatic.
$\tto^n_k$ is the statement ``Every coloring $f : [2^{<\omega}]^n \to k$ has an $f$-homogeneous tree.''
\end{definition}

Note that if $S \cong 2^{<\omega}$, witnessed by the bijection $g : 2^{<\omega} \to S$,
then $S$ is $g$-computable. Therefore we can consider that $\tto^n$ states the existence of the bijection $g$
instead of the pair~$\tuple{S, g}$.
The tree theorem was first analyzed by McNicholl~\cite{McNicholl1995Inclusion}
and by Chubb, Hirst, and McNicholl~\cite{Chubb2009Reverse}. They proved that~$\tto^2_2$
lies between $\aca$ and~$\rt^2_2$ over~$\rca$, and left open whether any of the implications
is strict. Further work was done by Corduan, Groszek, and Mileti~\cite{Corduan2010Reverse}.
Dzhafarov, Hirst and Lakins~\cite{Dzhafarov2010Ramseys} studied stability notions for the
tree theorem and introduced a polarized variant. Mont\'alban~\cite{Montalban2011Open} asked whether
$\rt^2_2$ implies~$\tto^2_2$ over~$\rca$. We give a negative answer by proving the following stronger theorem,
where $\wkl$ stands for weak K\"onig's lemma.

\begin{theorem}[Main result]
$\rt^2_2 \wedge \wkl$ does not imply~$\tto^2_2$ over~$\rca$.
\end{theorem}

The separation builds on the forcing method introduced by Lerman, Solomon and Towsner~\cite{Lerman2013Separating},
and enhanced by the author~\cite{Patey2015Iterative},
for designing a computability-theoretic property which discriminates between two statements in reverse
mathematics. The construction being quite complex, we present the proof step by step, hoping that our exposition
can serve as a tutorial for separating principles in reverse mathematics.

\subsection{Separating principles in reverse mathematics}\label{subsect:separating-rm}

An \emph{$\omega$-structure} is a structure $\Mcal = (\omega, \Scal, +, \cdot, <)$
where $\omega$ is the set of standard integers, $+$, $\cdot$ and $<$ 
are the standard operations over integers and $\Scal$ is a set of reals
such that~$\Mcal$ satisfies the axioms of~$\rca$.
Friedman~\cite{Friedman1974Some} characterized the second-order parts~$\Scal$ of~$\omega$-structures
as those forming a \emph{Turing ideal}, that is, a set of reals closed under Turing join
and downward-closed under the Turing reduction.
Given two principles~$\Psf$ and~$\Qsf$, proving that $\Psf$ does not imply~$\Qsf$
over~$\rca$ usually consists in constructing a Turing ideal~$\Ical$ in which~$\Psf$ holds,
but not~$\Qsf$.

Many theorems in reverse mathematics are~$\Pi^1_2$ statements, i.e., of the form $(\forall X)(\exists Y)\Phi(X, Y)$
where $\Phi$ is an arithmetic formula.
They can be considered as \emph{problems} which usually come with a natural
class of~\emph{instances}. Given an instance~$X$, a set~$Y$ such that~$\Phi(X,Y)$ holds
is called a~\emph{solution} to~$X$. For example, an instance of~$\rt^2_2$ is a coloring
$f : [\omega]^2 \to 2$, and a solution to~$f$ is an infinite $f$-homogeneous set.
In this setting, the construction of an~$\omega$-model of~$\Psf$ which is not a model of~$\Qsf$
consists in creating a Turing ideal~$\Ical$ together with a fixed $\Qsf$-instance $I \in \Ical$,
such that every $\Psf$-instance $J \in \Ical$ has a solution in~$\Ical$,
whereas $I$ contains no solution in~$\Ical$.
Building a Turing ideal is usually achieved via the following technique.

\begin{itemize}
  \item[1.] Choose a particular $\Qsf$-instance $I$ admitting no $I$-computable solution.
  \item[2.] Start with the Turing ideal $\Ical_0 = \{Z : Z \leq_T I\}$.
  \item[3.] Given a Turing ideal $\Ical_n$ containing no solution to $I$, 
  take any $\Psf$-instance $J \in \Ical_n$ having no solution in $\Ical_n$ and
  add a solution $Y$ to $J$. The closure by Turing reducibility and join gives $\Ical_{n+1}$.
  \item[4.] Repeat step 3 to obtain a Turing ideal $\Ical = \bigcup_n \Ical_n$
  such that every $\Psf$-instance in $\Ical$ admits a solution in $\Ical$.
\end{itemize}

The difficulty of such a construction is to avoid adding a solution to the instance $I$ in $\Ical_{n+1}$
during step~3. One needs to ensure that every $\Psf$-instance in $\Ical_n$ admits
a solution $Y$ such that $Y \oplus C$ \emph{avoids} computing a solution to $I$ for each $C \in \Ical_n$.

Assuming that $\Ical_n$ does not contain a solution to~$I$ is sometimes not sufficient
to ensure the existence of a solution~$Y$ to the next $\Psf$-instance such that 
the ideal closure of~$\Ical_n \cup \{Y\}$ does not contain such a solution as well.
A core step of the separation of~$\Psf$ from~$\Qsf$ consists in designing the
computability-theoretic property that we will propagate from~$\Ical_n$ to~$\Ical_{n+1}$
and which will ensure in particular that~$I$ has no solution in~$\Ical_{n+1}$.
This property strongly depends on the nature of the principles~$\Psf$ and~$\Qsf$.

Lerman, Solomon and Towsner~\cite{Lerman2013Separating} introduced a general technique
for designing such a property. Their framework has been successfuly reused
to separate various principles in reverse mathematics~\cite{Flood2014Separating,Patey2013note,Patey2015Ramsey}.
Recently, the author~\cite{Patey2015Iterative} refined their technique to make it more lightweight and modular.
Once simplified, a separation between two statements~$\Psf$ and~$\Qsf$ using the framework of Lerman, Solomon and Towsner
yields a computability-theoretic property called \emph{fairness property}. 
This property is closed downward under the Turing reduction
and is \emph{preserved} by the statement~$\Psf$, that is, for every \emph{fair} set~$X$ and every $X$-computable $\Psf$-instance~$I$,
there is a solution~$Y$ to~$I$ such that~$X \oplus Y$ is fair. This property is designed so that it is not preserved by~$\Qsf$,
which enables one to build an $\omega$-model of~$\Psf$ in which~$\Qsf$ does not hold.
Note that ``fairness property'' is a generic appelation for the computability-theoretic property yielded by the construction
of Lerman, Solomon and Towsner.  Different statements give different fairness properties,
such as hyperimmunity~\cite{Patey2015Iterative}, $\cac$-fairness~\cite{Patey2016reverse} or again $\tto$-fairness~\cite{Frittaion2015Coloring}. 
 
In this paper, we shall take the case of the separation of Ramsey's theorem for pairs from 
the tree theorem for pairs to make explicit the different steps leading to the separation
of two principles. In particular, we shall focus on the design of the discriminating property.

\subsection{Definitions and notation}

\emph{String, sequence}.
Fix an integer $k \in \omega$.
A \emph{string} (over $k$) of length~$n$ is an ordered tuple of integers $a_0, \dots, a_{n-1}$
(such that $a_i < k$ for every $i < n$). The empty string is written $\varepsilon$. A \emph{sequence}  (over $k$)
is an infinite listing of integers $a_0, a_1, \dots$ (such that $a_i < k$ for every $i \in \omega$).
Given $s \in \omega$,
$k^s$ is the set of strings of length $s$ over~$k$ and
$k^{<s}$ is the set of strings of length $<s$ over~$k$. Similarly,
$k^{<\omega}$ is the set of finite strings over~$k$
and $k^{\omega}$ is the set of sequences (i.e. infinite strings)
over~$k$. 
Given a string $\sigma \in k^{<\omega}$, we denote by $|\sigma|$ its length.
Given two strings $\sigma, \tau \in k^{<\omega}$, $\sigma$ is a \emph{prefix}
of $\tau$ (written $\sigma \preceq \tau$) if there exists a string $\rho \in k^{<\omega}$
such that $\sigma \rho = \tau$. Given a sequence $X$, we write $\sigma \prec X$ if
$\sigma = X \uh n$ for some $n \in \omega$, where $X \uh n$ denotes the restriction of $X$ to its first $n$ elements.
A \emph{binary string} (resp.\ binary sequence) is a \emph{string} (resp.\ sequence) over $2$.
We may identify a binary sequence with a set of integers by considering that the sequence is its characteristic function.

\emph{Tree, path}.
A tree $T \subseteq k^{<\omega}$ is a set downward-closed under the prefix relation.
A \emph{binary} tree is a tree $T \subseteq 2^{<\omega}$.
A sequence $P \in \omega^\omega$ is a \emph{path} though~$T$ if for every $\sigma \prec P$,
$\sigma \in T$. A string $\sigma \in k^{<\omega}$ is a \emph{stem} of a tree $T$
if every $\tau \in T$ is comparable with~$\sigma$.
Given a tree $T$ and a string $\sigma \in T$,
we denote by $T^{[\sigma]}$ the subtree $\{\tau \in T : \tau \preceq \sigma \vee \tau \succeq \sigma\}$.

\emph{Sets, partitions}.
Given two sets $A$ and $B$, we denote by $A < B$ the formula
$(\forall x \in A)(\forall y \in B)[x < y]$
and by $A \subseteq^{*} B$ the formula $(\forall^{\infty} x \in A)[x \in B]$,
meaning that $A$ is included in $B$ \emph{up to finitely many elements}.
Given a set~$X$ and some integer~$k$, a~\emph{$k$-cover of~$X$}
is a $k$-uple $A_0, \dots, A_{k-1}$ such that~$A_0 \cup \dots \cup A_{k-1} = X$.
We may simply say~\emph{$k$-cover} when the set~$X$ is unambiguous. 
A \emph{$k$-partition} is a $k$-cover whose sets are pairwise disjoint.
A \emph{Mathias condition} is a pair $(F, X)$
where $F$ is a finite set, $X$ is an infinite set
and $F < X$.
A condition $(F_1, X_1)$ \emph{extends } $(F, X)$ (written $(F_1, X_1) \leq (F, X)$)
if $F \subseteq F_1$, $X_1 \subseteq X$ and $F_1 \setminus F \subset X$.
A set $G$ \emph{satisfies} a Mathias condition $(F, X)$
if $F \subset G$ and $G \setminus F \subseteq X$.

%% file: parts/part1-partitions.tex
In order to get progressively into the framework
used to separate Ramsey's theorem for pairs from the tree theorem for pairs,
we shall first study the singleton version of the considered principles.
The fairness property that we shall design during the next section is directly
obtained by abstracting and generalizing the diagonalization argument of this section.
Ramsey's theorem for singletons is simply the infinite pigeonhole principle,
stating that for every finite partition of an infinite set,
one of its parts has an infinite subset.
Both $\rt^1_k$ and~$\tto^1_k$ are computably true and provable over~$\rca$.
We shall therefore study non-computable instances of~$\rt^1_k$ and~$\tto^1_k$
to see how their combinatorics differ.

As explained in subsection~\ref{subsect:separating-rm}, a proof of implication
from~$\Psf$ to~$\Qsf$ over~$\rca$ may involve multiple applications
of~$\Psf$. Therefore, if we want to prove that~$\Psf$ does not imply~$\Qsf$ over $\rca$,
we need to create an instance of~$\Qsf$ diagonalizing against
successive applications of~$\Psf$. In order to simplify our argument,
we shall first describe a one-step diagonalization between a $\Delta^0_2$ instance
of~$\tto^1_2$ and arbitrary instances of~$\rt^1_k$, that is, with no effectiveness restriction.
This is the notion of \emph{strong computable non-reducibility}.

\begin{definition}[Computable reducibility] Fix two~$\Pi^1_2$ statements~$\Psf$ and $\Qsf$.
$\Psf$ is \emph{strongly computably reducible} to a~$\Qsf$ (written~$\Psf \leq_{sc} \Qsf$)
if every~$\Psf$-instance~$I$ computes a~$\Qsf$-instance~$J$ such that every solution to~$J$
computes a solution to~$I$.
\end{definition}

The remainder of this section will be dedicated to proving that~$\tto^1_2 \not \leq_{sc} \rt^1_2$.
More precisely, we shall prove the following stronger theorem.

\begin{theorem}\label{thm:rt12-comp-not-reduc}
There exists a $\Delta^0_2$ $\tto^1_2$-instance $A_0 \cup A_1 = 2^{<\omega}$
such that for every (non-necessarily computable) $\rt^1_2$-instance $B_0 \cup B_1 = \omega$,
there is an infinite set homogeneous for the~$B$'s which does not compute
a $\tto^1_2$-solution to the~$A$'s.
\end{theorem}

In section~\ref{sect:separating-principles}, we will prove a theorem
which implies Theorem~\ref{thm:rt12-comp-not-reduc}. Therefore we shall focus on the key ideas
of the construction rather than on the technical details.

\smallskip
\emph{Requirements}. Let us first assume that we have constructed our $\tto^1_2$ instance~$A_0 \cup A_1 = 2^{<\omega}$.
Fix some 2-partition~$B_0 \cup B_1 = \omega$. We will construct by forcing an infinite set~$G$
such that both $G \cap B_0$ and~$G \cap B_1$ are infinite. Either $G \cap B_0$, or~$G \cap B_1$ will be taken
as our solution to the~$\rt^1_2$-instance $B_0 \cup B_1 = \omega$. We only need one solution to the~$\rt^1_2$-instance.
However, we will be only able to ensure that either $G \cap B_0$, or $G \cap B_1$ will not compute a solution to the 
$\tto^1_2$-instance~$A_0 \cup A_1 = 2^{<\omega}$. 
Therefore, among~$G \cap B_0$ and~$G \cap B_1$, only the one which does not compute a solution to the~$\tto^1_2$-instance
will be the desired solution to our $\rt^1_2$-instance.
Here, by ``solution to the~$\tto^1_2$-instance'', we mean an infinite subtree isomorphic to~$2^{<\omega}$ which is included
in~$A_0$ or~$A_1$.

Let~$\Phi_0, \Phi_1, \dots$ be an enumeration of all partial tree functionals isomorphic to $2^{<\omega}$, that is,
if $\Phi^X(n)$ halts, then $\Phi^X(n)$ outputs $2^n$ pairwise incomparable strings representing the $n$th level
of the tree. We ensure that the following requirements hold
for every pair of indices~$e_0,e_1$.
\[
  \Qcal_{e_0, e_1} : \hspace{20pt} 
		\Rcal_{e_0}^{G \cap B_0} \hspace{20pt} \vee \hspace{20pt}  \Rcal_{e_1}^{G \cap B_1}
\]
where $\Rcal_e^H$ is the statement
\begin{quote}
Either~$\Phi_e^H$ is partial, or~$\Phi_e^H(n)$ is a set~$D$ of $2^n$ incomparable
strings intersecting both~$A_0$ and~$A_1$ for some~$n$.
\end{quote}
We call any $\Qcal_{e_0, e_1}$ a \emph{$\Qcal$-requirement} and for a given~$H$ (being either~$G \cap B_0$, or $G \cap B_1$),
we call any~$\Rcal_e^H$ an \emph{$\Rcal$-requirement for~$H$}.
If every~$\Qcal$-requirement is satisfied, then by the usual pairing argument, either every~$\Rcal$-requirement
is satisfied for~$G \cap B_0$, or every $\Rcal$-requirement is satisfied for~$G \cap B_1$. 
Call $H$ a set among $G \cap B_0$ and~$G \cap B_1$ for which every $\Rcal$-requirement is satisfied.
We claim that $H$ does not compute a solution to the~$\tto^1_2$-instance $A_0 \cup A_1 = 2^{<\omega}$.
Suppose for the sake of contradiction that $H$ computes a tree~$S \cong 2^{<\omega}$ using some procedure~$\Phi_e$.
By the requirement $\Rcal_e^H$, $S$ intersects both~$A_0$ and~$A_1$, and therefore $S$ is not a $\tto^1_2$-solution to the~$A$'s.

\smallskip
\emph{Forcing}. The forcing conditions are Mathias conditions, that is, an ordered pair~$(F, X)$, 
where~$F$ is a finite set of integers, $X$ is an infinite set belonging to some fixed Scott set $\Scal$,
and such that $max(F) < min(X)$.
A \emph{Scott set} Turing ideal satisfying weak König's lemma. By Simpson~\cite[Theorem~VIII.2.17]{Simpson2009Subsystems}, we can
choose~$\Scal$ so that $\Scal = \{ X_i : C = \bigoplus_i X_i \}$ for some low set $C$. This precision will be useful
during the construction of the~$\tto^1_2$-instance. We furthermore assume that~$C$ does not compute a $\tto^1_2$-solution to the~$A$'s,
and therefore that there is no $C$-computable infinite set homogeneous for the~$B$'s, otherwise we are done.

The following lemma ensures that we can force both~$G \cap B_0$ and~$G \cap B_1$ to be infinite,
assuming that the $B$'s have no infinite $C$-computable homogeneous set.

\begin{lemma}\label{lem:rt12-comp-reduc-force-N}
Given a condition~$c = (F, X)$ and some side~$i < 2$, 
there is an extension~$d = (E, Y)$ such that~$|E \cap B_i| > |F \cap B_i|$.
\end{lemma}
\begin{proof}
If $X \cap B_i = \emptyset$ then $X$ is an infinite $C$-computable subset of~$B_{1-i}$, contradicting
our assumption. So there is some~$x \in X \cap B_i$. Take~$d = (F \cup \{x\}, X \setminus [0,x])$
as the desired extension.
\end{proof}

The next step consists in forcing the~$\Qcal$-requirements to be satisfied.
A condition~$c$ \emph{forces} a requirement~$\Qcal_{e_0,e_1}$ if $\Qcal_{e_0,e_1}$
holds for every set~$G$ satisfying~$c$. Of course, we cannot force the~$\Qcal$-requirements 
for any $\tto^1_2$-instance~$A_0 \cup A_1 = 2^{<\omega}$ since some of them admit computable solutions.
We must therefore choose our $\tto^1_2$-instance carefully. For now, simply assume that we managed
to build a $\tto^1_2$-instance $A_0 \cup A_1 = 2^{<\omega}$ satisfying the following property (P).
We will detail its construction later.

\begin{quote}(P) %\label{lem:rt12-comp-reduc-forcing-Q}
Given a condition~$c = (F, X)$ and some indices~$e_0,e_1$,
there is an extension~$d$ of~$c$ forcing~$\Qcal_{e_0,e_1}$.
\end{quote}

Assuming that the property (P) holds, we now show how to build
our infinite set~$G$ from it. After that, we will construct a $\tto^1_2$-instance~$A_0 \cup A_1 = 2^{<\omega}$
so that the property (P) is satisfied.

\smallskip
\emph{Construction}.
Thanks to Lemma~\ref{lem:rt12-comp-reduc-force-N} and the property (P), 
we can define an infinite, decreasing sequence of conditions~$(\emptyset, \omega) \geq c_0 \geq c_1 \dots$
such that for each~$s \in \omega$
\begin{itemize}
	\item[(i)] $|F_s \cap B_0| \geq s$ and~$|F_s \cap B_1| \geq s$
	\item[(ii)] $c_s$ forces~$\Qcal_{e_0,e_1}$ if~$s = \tuple{e_0,e_1}$
\end{itemize}
where~$c_s = (F_s, X_s)$. The set~$G = \bigcup_s F_s$ is such that
both $G \cap B_0$ and~$G \cap B_1$ are infinite by (i),
and either~$G \cap B_0$ or~$G \cap B_1$ does not compute a $\tto^1_2$-solution
to the~$A$'s by~(ii). The set among~$G \cap B_0$ and~$G \cap B_1$ which does not compute a $\tto^1_2$-solution
to the~$A$'s is our desired $\rt^1_2$-solution to the~$B$'s. We now need to satisfy the property (P).

\smallskip
\emph{Satisfying the property (P)}.
Given a condition, the extension stated in the property (P)
cannot be ensured for an arbitrary $\tto^1_2$-instance~$A_0 \cup A_1 = 2^{<\omega}$.
We must design the $A$'s so that the property (P) holds.
To do so, we will apply the ideas developped by Lerman, Solomon and Towsner~\cite{Lerman2013Separating}.
We can see the construction of the set~$G$ as a game. The opponent is the~$\tto^1_2$-instance which
will try everything, not to be diagonalized against. However, the opponent is \emph{fair}, in the sense that
if we have infinitely many occasions to diagonalize against him, then he will let us do it.
More precisely, if given a condition~$c = (F, X)$ and some indices~$e_0,e_1$, we can find extensions
which makes both~$\Phi^{G \cap B_0}_{e_0}$ and~$\Phi^{G \cap B_1}_{e_1}$ produce arbitrarily large initial segments of the output, 
then one of those outputs will intersect both~$A_0$ and~$A_1$. In this case, we will have succeeded to satisfy~(P) for the condition~$c$
by producing an extension~$d$ forcing $\Phi^{G \cap B_i}_{e_i}$ to intersect both~$A_0$ and~$A_1$ for some~$i < 2$,
and therefore forcing~$\Qcal_{e_0, e_1}$. In the other case, we will also have vacuously succeeded since
we will not be able to find extensions making both $\Phi^{G \cap B_0}_{e_0}$ and $\Phi^{G \cap B_1}_{e_1}$ produce arbitrarily large initial segments of the output tree, and therefore $c$ is already a condition forcing one $\Phi^{G \cap B_0}_{e_0}$ 
and $\Phi^{G \cap B_1}_{e_1}$ to be partial, so forcing~$\Qcal_{e_0, e_1}$.

We now describe how to construct a fair $\tto^1_2$-instance. The construction of the~$A$'s
will be $\Delta^{0,C}_2$, hence~$\Delta^0_2$ since~$C$ is low. The access to the oracle~$C$ enables us to
\emph{code} the conditions $c = (F, X)$ into finite objects, namely, pairs~$(F, i)$ so that~$C = \bigoplus_i X_i$ and $X = X_i$,
and to enumerate them $C$-computably. More precisely, we will enumerate all 6-tuples
$\tuple{F, i, E_0, E_1, e_0, e_1}$, where~$(F, X_i)$ is a \emph{preconditions}, that is, a condition where
we drop the constraint that~$X_i$ is infinite since it requires too much computational power to know it,
$E_0 \sqcup E_1 = F$ represents a \emph{guess} of the sets~$F \cap B_0$ and~$F \cap B_1$,
and~$e_0, e_1$ denote the~$\Qcal_{e_0, e_1}$ we want to force.
In particular, among those 6-tuples enumerated, we will enumerate $\tuple{F, i, F \cap B_0, F \cap B_1, e_0, e_1}$
for all the true conditions $(F, X_i)$.

The construction of the~$A$'s is done by stages.
At stage~$s$, we have constructed two sets~$A_{0,s} \cup A_{1,s} = 2^{<q}$ for some~$q \in \omega$.
We want to satisfy the property (P) given a 6-tuple $\tuple{F, i, E_0, E_1, e_0, e_1}$,
that is, given a precondition~$c = (F, X)$, a guess of $F \cap B_0$ and $F \cap B_1$, and a pair of indices~$e_0, e_1$.
If any of $\Phi^{E_0}_{e_0}(2)$ and $\Phi^{E_1}_{e_1}(2)$ is not defined, do nothing and go to the next stage.
We can restrict ourselves without loss of generality to preconditions such that both $\Phi^{E_0}_{e_0}(2)$
and $\Phi^{E_1}_{e_1}(2)$ are defined. Indeed, if in the property (P),
the condition~$c$ has no such extension, then~$c$ already forces
either $\Phi^{G \cap B_0}_{e_0}$ or $\Phi^{G \cap B_1}_{e_1}$ to be partial and therefore
vacuously forces~$\Qcal_{e_0,e_1}$. The choice of ``2'' as input seems arbitrary. It has not been picked
randomly and this choice will be justified in the next paragraph.

Let~$D_0$ and $D_1$ be the 4-sets of pairwise incomparable strings outputted by 
$\Phi^{E_0}_{e_0}(2)$ and $\Phi^{E_1}_{e_1}(2)$, respectively.
Altough the strings are pairwise incomparable \emph{within} $D_0$ or $D_1$,
there may be two comparable strings in $D_0 \cup D_1$. However,
by a simple combinatorial argument, we may always find two strings~$\sigma_0, \tau_0 \in D_0$ and
$\sigma_1, \tau_1 \in D_1$ such that $\sigma_0, \tau_0, \sigma_1$ and $\tau_1$
are pairwise incomparable (see Lemma~\ref{lem:matrix-combi}). Here, we use the fact that
on input 2, the sets have cardinality 4, which is enough to apply Lemma~\ref{lem:matrix-combi}.
We are now ready to ask the main question.

\smallskip
``Is it true that for every $2$-partition~$Z_0 \cup Z_1 = X$, there is some side~$i < 2$
and some set~$G \subseteq Z_i$ such that $\Phi^{E_i \cup G}_{e_i}(q)$ halts?''
\smallskip

Note that the question looks $\Pi^{1,X}_2$, but is in fact $\Sigma^{0,X}_1$ by a compactness argument.
It is therefore $C'$-decidable since~$X \in \Scal$ and so can be uniformly decided during the construction.
We have two cases.

Case 1: The answer is negative.
In this case, the $\Pi^{0,X}_1$ class $\Ccal$ of all sets $Z_0 \oplus Z_1$
such that $Z_0 \cup Z_1 = X$ and for every~$i < 2$ and every set~$G \subseteq Z_i$,
$\Phi^{E_i \cup G}_{e_i}(q) \uparrow$ is non-empty. In this case,
we do nothing and claim that the property (P) holds for~$c$.
Indeed, since~$\Scal$ is a Scott set containing $X$, there is some~$Z_0 \oplus Z_1 \in \Ccal \cap \Scal$
such that $Z_0 \cup Z_1 = X$. As $X$ is infinite, there is some~$i < 2$ such that~$Z_i$ is infinite.
In this case, if~$E_0 = F \cap B_0$ and~$E_1 = F \cap B_1$,
$d = (F, Z_i)$ is an extension forcing~$\Phi^{(G \cap B_i)}_{e_i}(q) \uparrow$
and therefore forcing~$\Qcal_{e_0,e_1}$. Note that this extension
cannot be found $C'$-computably since it requires to decide which of~$Z_0$ and~$Z_1$ is infinite.
However, we do not need to uniformly provide this extension. The property (P) 
simply states the \emph{existence} of such an extension.

Case 2: The answer is positive. Given a string~$\sigma \in 2^{<\omega}$,
let~$S_\sigma = \{\tau \succeq \sigma \}$. 
Since the $\Phi$'s are tree functionals and $\Phi^{E_i \cup G}_{e_i}(2)$
outputs (among others) the strings~$\sigma_i$ and~$\tau_i$, whenever $\Phi^{E_i \cup G}_{e_i}(q)$
halts, it outputs a finite set $D$ of size $2^q$ intersecting both~$S_{\sigma_i}$ and~$S_{\tau_i}$.
Therefore, by compactness, there are finite sets $U_0 \subseteq S_{\sigma_0}$,
$V_0 \subseteq S_{\tau_0}$,  $U_1 \subseteq S_{\sigma_1}$ and~$V_1 \subseteq S_{\tau_1}$
such that for every $2$-partition~$Z_0 \cup Z_1 = X$, there is some side~$i < 2$
and some set~$G \subseteq Z_i$ such that $\Phi^{E_i \cup G}_{e_i}(q)$ intersects both $U_i$ and~$V_i$.
In particular, whenever $E_0 = F \cap G_0$, $E_1 = F \cap B_1$, $Z_0 = X \cap B_0$ and~$Z_1 = X \cap B_1$,
there is some $i < 2$ and some finite set~$G \subseteq X \cap B_i$ such that 
$\Phi^{(F \cap B_i) \cup G}_{e_i}(q)$ intersects both $U_i$ and~$V_i$.
Notice that all the strings in $U_i$ and $V_i$ have length at least~$q$
and therefore are not yet colored by the~$A$'s.
Put the~$U$'s in~$A_{0,s+1}$ and the~$V$'s in~$A_{1,s+1}$ and complete the coloring
so that~$A_{0,s+1} \cup A_{1, s+1} = 2^{<r}$ for some~$r \geq q$. Then go to the next step.
We claim that the property (P) holds for~$c$.
Indeed, let $G \subseteq X \cap B_i$ be the finite set witnessing 
$\Phi^{(F \cap B_i) \cup G}_{e_i}(q) \cap U_i \neq \emptyset$ and $\Phi^{(F \cap B_i) \cup G}_{e_i}(q) \cap V_i \neq \emptyset$.
The condition~$d = (F \cup G, X \setminus [0, max(E)])$ is an extension forcing~$\Qcal_{e_0,e_1}$
by its $i$th side.
This finishes the construction of the~$\tto^1_2$-instance and the proof of Theorem~\ref{thm:rt12-comp-not-reduc}.

%% file: parts/part2-fairness.tex
In this section, we analyse the one-step separation proof
of section~\ref{sect:partitions-reducibility} in order to extract the core of the argument.
Then, we use the framework of Lerman, Solomon and Towsner
to design the computability-theoretic property
which will enable us to discriminate~$\rt^2_2$ from~$\tto^2_2$.

\smallskip
\emph{The multiple-step case}.
We have seen how to diagonalize against one application of~$\rt^1_2$.
The strength of~$\tto^1_2$ comes from the fact that 
when we build a solution $S$ to some $\tto^1_2$-instance $A_0 \cup A_1 = 2^{<\omega}$,
we must provide finite subtrees $S_n \cong 2^{<n}$ for arbitrarily large~$n$.
However, as soon as we have outputted $S_n$, we \emph{commit}
to provide arbitrarily large extensions to each leaf of~$S_n$. Since the
leaves in $S_n$ are pairwise incomparable, the sets of their extensions
are mutually disjoint. During the construction of the $\tto^1_2$-instance,
we can pick any pair $\sigma, \tau$ of incomparable leaves
in~$S_n$, and put the extensions of~$\sigma$ in $A_0$ and the extensions of~$\tau$ in~$A_1$
since they are disjoint.

In the proof of $\tto^1_2 \not \leq_{sc} \rt^1_2$,
when we create a solution to some~$\rt^1_2$-instance~$B_0 \cup B_1 = \omega$,
we build two candidate solutions~$G \cap B_0$ and~$G \cap B_1$ at the same time.
For each pair of tree functionals $\Phi_{e_0}$ and~$\Phi_{e_1}$, 
we must prevent one of $\Phi^{G \cap B_0}_{e_0}$ and~$\Phi^{G \cap B_1}_{e_1}$
from being a $\tto^1_2$-solution to the~$A$'s. However, the finite
subtrees $S_0$ and~$S_1$ outputted respectively by the left side and the right side
may have comparable leaves. We cannot take any 2 leaves of~$S_0$ and 2 leaves of~$S_1$
to obtain 4 pairwise incomparable strings. Thankfully, if $S_0$ and $S_1$ contain enough leaves (4 is enough),
we can find such strings.

If we try to diagonalize against two applications of~$\rt^1_2$,
below each side $G \cap B_0$ and~$G \cap B_1$ of the first $\rt^1_2$-instance,
we will have again two sides corresponding to the second $\rt^1_2$-instance.
We will have then to diagonalize against four candidate subtrees $S_0$, $S_1$, $S_2$ and~$S_3$.
We need therefore to wait until the subtrees have enough leaves, so that we can find 8 pairwise
incomparable leaves $\sigma_0, \tau_0, \dots, \sigma_3, \tau_3$
such that~$\sigma_i, \tau_i \in S_i$ for each~$i < 4$.

In the general case, we will then have to diagonalize against an arbitrarily large number of subtrees,
and want to ensure that if they contain enough leaves,
we can find two leaves in each, such that they form a set of pairwise incomparable strings.
This leads to the notion of disjoint matrix.

\begin{definition}[Disjoint matrix]
An $m$-by-$n$ \emph{matrix} $M$ is a rectangular array of strings $\sigma_{i,j} \in 2^{<\omega}$
such that $i < m$ and $j < n$. The $i$th \emph{row} $M(i)$ of the matrix $M$
is the $n$-tuple of strings $\sigma_{i,0}, \dots, \sigma_{i,n-1}$.
An $m$-by-$n$ matrix $M$ is \emph{disjoint} if for each row $i < m$,
the strings $\sigma_{i,0}, \dots, \sigma_{i,n-1}$ are pairwise incomparable.
\end{definition}

The following combinatorial lemma gives an explicit bound on the number of leaves 
we require on each subtree to obtain our desired sequence of
pairwise incomparable strings.

\begin{lemma}\label{lem:matrix-combi}
For every $m$-by-$2m$ disjoint matrix $M$, there are
pairwise incomparable strings $\sigma_0, \tau_0, \dots, \allowbreak \sigma_{m-1}, \tau_{m-1}$
such that $\sigma_i, \tau_i \in M(i)$ for every $i < m$.
\end{lemma}
\begin{proof}
Consider the following greedy algorithm.
At each stage, we maintain a set~$P$ of \emph{pending rows} which is initially the whole matrix~$M$.
Among those rows, some strings are flagged as \emph{invalid}. Initially, all the strings are valid.
Pick a string $\rho$ of maximal length among all the valid strings of all the pending rows.
Let $M(i)$ be a pending row to which $\rho$ belongs.
If we have already chosen the value of $\sigma_i$, set $\tau_i = \rho$ and remove 
$M(i)$ from the pending rows. Otherwise, set $\sigma_i = \rho$.
In any case, flag every prefix of $\rho$ from any row of $M$ as invalid, and go to the next step.

Notice that at any step, we flag as invalid at most one string from each row of $M$
since the strings in each row are pairwise incomparable.
Moreover, since we want to construct a sequence of $2m$
pairwise incomparable strings, and at each step we add one string to this sequence,
there are at most $2m$ steps. The algorithm gets stuck at some points only if all pending rows contain only invalid strings,
which cannot happen since each row contains at least $2m$ strings.

We claim that the chosen strings are pairwise incomparable. Indeed, when at some stage, we pick
a string $\rho$, it is of shorter length than any string we have picked so far,
and cannot be a prefix of any of them since each time we pick a string, we flag all its prefixes as invalid in the matrix.
\end{proof}

\emph{Abstracting the requirements}.
The first feature of the framework of Lerman, Solomon and Towsner 
that we already exploited is the ``fairness'' of the $\tto^1_2$-instance
which allows each~$\rt^1_2$-instance to diagonalize it
as soon as the $\rt^1_2$-instance gives him enough occasions to do it.
We will now use the second aspect of this framework which consists in
getting rid of the complexity of the requirements by replacing them
with arbitrary computable predicates (or blackboxes).

Indeed, consider the case of two successive applications of~$\rt^1_2$.
Say that the first instance is $B_0 \cup B_1 = \omega$, and the second $C_0 \cup C_1 = \omega$.
We need to design the $\tto^1_2$-instance $A_0 \cup A_1 = 2^{<\omega}$ so that
there is an infinite set $G \cap B_i$ for some~$i < 2$ and an infinite set~$H \cap B_j$
for some~$j < 2$ such that $(G \cap B_i) \oplus (H \cap B_j)$ does not compute a
solution to the~$A$'s.
While constructing the $A$'s, we enumerate two levels of conditions.
We first enumerate the conditions~$c = (F, X)$ used for constructing the set~$G$,
but we also enumerate the conditions $c_0 = (F_0, X_0)$ and~$c_1 = (F_1, X_1)$
such that~$c_i$ is used to construct a solution~$H$ to the second $\rt^1_2$-instance~$C_0 \cup C_1 = \omega$
below~$G \cap B_i$. The question that the $\tto^1_2$-instance asks during its construction becomes

\smallskip
``For every 2-partition~$Z_0 \cup Z_1 = X$, is there some side~$i < 2$
and some set~$G \subseteq Z_i$ such that for every 2-partition~$W_0 \cup W_1 = X_i$,
there is some side~$j < 2$ and some set~$H \subseteq W_j$ 
such that $\Phi_{e_{i,j}}^{((F \cap B_i) \cup G) \oplus ((F_i \cap C_j) \cup H)}(q)$ halts?''
\smallskip

While staying~$\Sigma^0_1$ (with parameters), the question becomes
arbitrarily complicated to formulate. Moreover, looking at the shape of the question,
we see that the first iteration can box any $\Sigma^0_1$ question asked about the second iteration.
We can therefore abstract the question and make the fairness property independent
of the specificities of the forcing notion used to solve the~$\rt^1_2$-instances.
See~\cite{Patey2015Iterative} for detailed explanations about this abstraction process.

\begin{definition}[Formula, valuation]
An \emph{$m$-by-$n$ formula} is a formula $\varphi$
with distinguished set variables $U_{i,j}$ for each $i < m$ and $j < n$.
Given an $m$-by-$n$ matrix $M = \{\sigma_{i,j} : i < m, j < n\}$, 
an \emph{$M$-valuation} $V$ is a tuple of finite sets 
$A_{i,j} \subseteq \{\tau \in 2^{<\omega} : \tau \succeq \sigma_{i,j}\}$ 
for each $i < m$ and~$j < n$.
The valuation $V$ \emph{satisfies} $\varphi$ if $\varphi(A_{i,j} : i < m, j < n)$ holds.
We write $\varphi(V)$ for $\varphi(A_{i,j} : i < m, j < n)$.
\end{definition}

Given some valuation $V = (A_{i,j} : i < m, j < n)$ and some integer $s$, we write $V > s$
to say that for every $\tau \in A_{i,j}$, $|\tau| > s$. Moreover,
we denote by $V(i)$ the $n$-tuple $A_{i,0}, \dots, A_{i,n-1}$.
Following the terminology of~\cite{Lerman2013Separating}, we define 
the notion of essentiality for a formula (an abstract requirement),
which corresponds to the idea that there is room for diagonalization
since the formula is satisfied for arbitrarily far valuations.

\begin{definition}[Essential formula]
An $m$-by-$n$ formula $\varphi$ is \emph{essential} in an $m$-by-$n$ matrix $M$
if for every $s \in \omega$, there is an $M$-valuation $V > s$ such that $\varphi(V)$ holds.
\end{definition}

The notion of fairness is defined accordingly. If some formula
is essential, that is, gives enough room for diagonalization, then there is
an actual valuation which will diagonalize against the~$\tto^1_2$-instance.

\begin{definition}[Fairness]
Fix two sets $A_0, A_1 \subseteq 2^{<\omega}$.
Given an $m$-by-$n$ disjoint matrix $M$, an $M$-valuation~$V$ \emph{diagonalizes} against $A_0, A_1 \subseteq 2^{<\omega}$
if for every $i < m$, there is some~$L, R \in V(i)$ such that $L \subseteq A_0$ and $R \subseteq A_1$.
A set~$X$ is \emph{$n$-fair} for~$A_0, A_1$ if for every $m$ and every $\Sigma^{0,X}_1$ $m$-by-$2^nm$ formula
$\varphi$ essential in some disjoint matrix $M$, there is an $M$-valuation $V$ diagonalizing against $A_0, A_1$ 
such that $\varphi(V)$ holds.
A set $X$ is \emph{fair} for $A_0, A_1$ if it is $n$-fair for~$A_0, A_1$ for some~$n \geq 1$.
\end{definition}

Of course, if $Y \leq_T X$, then every $\Sigma^{0,Y}_1$ formula is $\Sigma^{0,X}_1$.
As an immediate consequence, if $X$ is $n$-fair for some $A_0, A_1$ and $Y \leq_T X$, then $Y$ is $n$-fair for $A_0, A_1$.
Moreover, if $X$ is $n$-fair for $A_0, A_1$ and $p > n$, $X$ is also $p$-fair for $A_0, A_1$ as witnessed
by cropping the rows.

\begin{definition}[Fairness preservation]
Fix a $\Pi^1_2$ statement $\Psf$.
\begin{itemize}
	\item[1.] $\Psf$ admits \emph{fairness (resp. $n$-fairness) preservation} if for all sets $A_0, A_1 \subseteq 2^{<\omega}$,
every set $C$ which is fair (resp. $n$-fair) for~$A_0, A_1$ and every $C$-computable $\Psf$-instance~$X$,
there is a solution $Y$ to $X$ such that $Y \oplus C$ is fair (resp. $n$-fair) for $A_0, A_1$.
	\item[2.] $\Psf$ admits strong \emph{fairness (resp. $n$-fairness) preservation} if for all sets $A_0, A_1 \subseteq 2^{<\omega}$,
every set $C$ which is fair (resp. $n$-fair) for~$A_0, A_1$ and every $\Psf$-instance~$X$,
there is a solution $Y$ to $X$ such that $Y \oplus C$ is fair (resp. $n$-fair) for $A_0, A_1$.
\end{itemize}
\end{definition}

Note that a principle $\Psf$ may admit fairness preservation without preserving $n$-fairness for any fixed~$n$,
as this is the case with~$\rt^2_2$ (see Theorem~\ref{thm:rt22-fairness-preservation}
and Theorem~\ref{thm:rt22-not-1-fairness-preservation}). 
On the other hand, if $\Psf$ admits $n$-fairness preservation for every $n$,
then it admits fairness preservation.
The notion of fairness preservation has been designed so that it is closed under the implication
over~$\rca$.

\begin{lemma}\label{lem:fairness-provability}
If $\Psf$ admits fairness preservation but not $\Qsf$, then $\Psf$ does not imply $\Qsf$
over $\rca$.
\end{lemma}
\begin{proof}
Since $\Qsf$ does not admit fairness preservation, 
there is a set $C$ which is $n$-fair for some $A_0, A_1 \subseteq 2^{<\omega}$ and 
a $C$-computable $\Qsf$-instance~$J$
such that for every solution $Y$ to $J$, $C \oplus Y$ is not fair for $A_0, A_1$.
We build an infinite sequence of sets $X_0, X_1, \dots$ starting with $X_0 = C$
and such that for every $s \in \omega$,
\begin{itemize}
	\item[(i)] $X_{s+1}$ is a solution to the $\Psf$-instance $I^{X_0 \oplus \dots \oplus X_s}_s$
	\item[(ii)] $X_0 \oplus \dots \oplus X_{s+1}$ is fair for $A_0, A_1$
\end{itemize}
where $I_0, I_1, \dots$ is a (non-computable) enumeration of all $\Psf$-instance functionals.
Let $\Mcal$ be the $\omega$-model whose second-order part is the Turing ideal
\[
  \Ical = \{ Z : (\exists s)[Z \leq_T X_0 \oplus \dots \oplus X_s] \}
\]
By Friedman~\cite{Friedman1974Some}, $\Mcal \models \rca$ since $\Ical$ is a Turing ideal.
By (i), $\Mcal \models \Psf$. As $J \leq_T C = X_0$, $J \in \Ical$.
However, for every $Z \in \Ical$, $Z \oplus C$ is fair for $A_0, A_1$ by downward closure
of fairness under the Turing reducibility, so $Z$ is not a solution to the $\Qsf$-instance $J$.
Therefore $\Mcal \not \models \Qsf$.
\end{proof}

Now we have introduced the necessary terminology, we create a $\Delta^0_2$
instance of~$\tto^1_2$ which will serve as a bootstrap for fairness preservation.

\begin{lemma}\label{lem:partition-emptyset-fair}
There exists a $\Delta^0_2$ partition $A_0 \cup A_1 = 2^{<\omega}$ such that
$\emptyset$ is 1-fair for~$A_0, A_1$.
\end{lemma}
\begin{proof}
The proof is done by a no-injury priority construction.
Let $\varphi_0, \varphi_1, \dots$ be a computable enumeration of all $m$-by-$2m$ $\Sigma^0_1$ formulas
and $M_0, M_1, \dots$ be an enumeration of all $m$-by-$2m$ disjoint matrices for every $m$.
We want to satisfy the following requirements for each pair of integers~$e,k$.

\begin{quote}
$\Rcal_{e,k}$: If $\varphi_e$ is essential in $M_k$, then $\varphi_e(V)$ holds
for some $M_k$-valuation $V$ diagonalizing against $A_0, A_1$.
\end{quote}

The requirements are ordered via the standard pairing function $\tuple{\cdot, \cdot}$.
The sets $A_0$ and $A_1$ are constructed by a $\emptyset'$-computable list of 
finite approximations $A_{i,0} \subseteq A_{i,1} \subseteq \dots$
such that all elements added to~$A_{i,s+1}$ from~$A_{i,s}$
are strictly greater than the maximum of~$A_{i,s}$ for each~$i < 2$. 
We then let $A_i = \bigcup_s A_{i,s}$ which will be a~$\Delta^0_2$ set.
At stage 0, set $A_{0,0} = A_{1,0} = \emptyset$. Suppose that at stage $s$,
we have defined two disjoint finite sets $A_{0,s}$ and $A_{1,s}$ such that
\begin{itemize}
	\item[(i)] $A_{0,s} \cup A_{1,s} = 2^{<b}$ for some integer $b \geq s$
	\item[(ii)] $\Rcal_{e',k'}$ is satisfied for every $\tuple{e',k'} < s$
\end{itemize}
Let $\Rcal_{e,k}$ be the requirement such that $\tuple{e,k} = s$.
Decide $\emptyset'$-computably whether there is some $M_k$-valuation $V > b$
such that $\varphi_e(V)$ holds. If so, computably fetch such a~$V$ 
and let $d$ be an upper bound on the length of the strings in $V$.
By Lemma~\ref{lem:matrix-combi}, there are
pairwise incomparable strings $\sigma_0, \tau_0, \dots, \allowbreak \sigma_{m-1}, \tau_{m-1}$
such that $\sigma_i, \tau_i \in M(i)$ for every $i < m$.
For each $i < m$, let $A_{i,l}$ and $A_{i,r}$ be the sets in $V$ corresponding to $\sigma_i$
and $\tau_i$, respectively.
Set $A_{0,s+1} = A_{0,s} \bigcup_{i < m} A_{i,l}$ and $A_{1,s+1} = 2^{<d} \setminus A_{0,s+1}$.
This way, $A_{0,s+1} \cup A_{1,s+1} = 2^{<d}$.
Since the $\sigma$'s and $\tau$'s are pairwise incomparable, the sets $A_{i,l}$ and $A_{i,r}$ are disjoint,
so $\bigcup_{i < m} A_{i,r} \subseteq 2^{<d} \setminus A_{0,s+1}$ and the requirement $\Rcal_{e,i}$ is satisfied.
If no such $M_k$-valuation is found, the requirement $\Rcal_{e,k}$ is vacuously satisfied.
Set $A_{0,s+1} = A_{0,s} \cup 2^b$ and $A_{1,s+1} = A_{1,s}$.
This way, $A_{0,s+1} \cup A_{1,s+1} = 2^{<(b+1)}$.
In any case, go to the next stage. This finishes the construction.
\end{proof}

\begin{theorem}\label{thm:tt22-not-fairness}
$\tto^2_2$ does not admit fairness preservation.
\end{theorem}
\begin{proof}
Let~$A_0 \cup A_1 = 2^{<\omega}$ be the $\Delta^0_2$ partition constructed in Lemma~\ref{lem:partition-emptyset-fair}.
By Schoenfield's limit lemma~\cite{Shoenfield1959degrees}, there is a computable function $h : 2^{<\omega} \times \omega \to 2$ such 
that for each~$\sigma \in 2^{<\omega}$, $\lim_s h(\sigma, s)$ exists and $\sigma \in A_{\lim_s h(\sigma, s)}$.
Let~$f : [2^{<\omega}]^2 \to 2$ be the computable coloring defined by $f(\sigma, \tau) = h(\sigma, |\tau|)$
for each~$\sigma \prec \tau \in 2^{<\omega}$. Let~$S \cong 2^{<\omega}$ be a $\tto^2_2$-solution to~$f$
with witness isomorphism~$g : 2^{<\omega} \to S$ and witness color~$c < 2$. Note that $S \subseteq A_c$.

Fix any~$n \geq 1$. 
We claim that $S$ is not $n$-fair for~$A_0, A_1$. For this, we construct a $1$-by-$2^n$ $\Sigma^{0,S}_1$ formula
and a $1$-by-$2^n$ disjoint matrix~$M$ such that~$\varphi$ is essential in~$M$,
but such that every $M$-valuation~$V$ satisfying $\varphi$ is included in~$A_c$.

Let~$\varphi(U_j : j < 2^n)$ be the $1$-by-$2^n$ $\Sigma^{0,S}_1$ formula
which holds if for each $j < 2^n$, $U_j$ is a non-empty subset of $S$.
Let $M = (\sigma_j : j < 2^n)$ be the $1$-by-$2^n$ disjoint matrix defined 
for each $j < 2^n$ by
$\sigma_j = g(\tau_j)$ where $\tau_j$ is the $j$th string of length~$n$.
In other words, $\sigma_j$ is the $j$th node at level~$n$ in~$S$.
For every $s$, let $V_s$ be the $M$-valuation defined by $B_j = \{g(\rho)\}$ such that $\rho$ is the least string of length $max(n,s)$ extending~$\tau_j$.  Notice that $V_s > s$ and $\varphi(V_s)$ holds.
Therefore, the formula $\varphi$ is essential in $M$. 
For every $M$-valuation $V = (B_j : j < 2^n)$ such that $\varphi(V)$ holds,
there is no $j < 2^n$ such that $B_j \subseteq A_{1-c}$. Indeed, since $\varphi(V)$ holds,
$B_j$ is a non-empty subset of~$S$, which is itself a subset of $A_c$.
Therefore $S$ is not $n$-fair for~$A_0, A_1$.
\end{proof}

Notice that we actually proved a stronger statement.
Dzhafarov, Hirst and Lakins defined in~\cite{Dzhafarov2010Ramseys} various notions
of stability for the tree theorem for pairs. A coloring~$f : [2^{<\omega}]^2 \to r$
is \emph{1-stable} if for every $\sigma \in 2^{<\omega}$, there is some threshold $t$
and some color~$c < r$ such that $f(\sigma, \tau) = c$ for every $\tau \succ \sigma$
such that $|\tau| \geq t$. In the proof of Theorem~\ref{thm:tt22-not-fairness}, we showed in fact that
$\tto^2_2$ restricted to 1-stable colorings does not admit fairness preservation.
In the same paper, Dzhafarov et al.\ studied an increasing polarized version of the tree theorem for pairs,
and proved that its 1-stable restriction coincides with the 1-stable tree theorem for pairs over~$\rca$.
Therefore the increasing polarized tree theorem for pairs does not admit fairness preservation.

%% file: parts/part3-separations.tex
In this section, we prove fairness preservation for various principles in reverse mathematics,
namely, weak K\"onig's lemma, cohesiveness and $\rt^2_2$. We prove independently that
they admit fairness preservation, and then use the compositional
nature of the notion of preservation to deduce that the conjunction of these
principles do not imply $\tto^2_2$ over $\rca$.

\begin{definition}[Weak König's lemma]
$\wkl$ is the statement ``Every infinite binary tree has an infinite path''.
\end{definition}

Weak König's lemma is one of the ``Big Five''. 
It can be thought of as capturing compactness arguments.
The question of its relation with $\rt^2_2$ has been a long standing open problem,
until Cholak, Jockusch and Slaman~\cite{Cholak2001strength} 
and Liu~\cite{Liu2012RT22} proved that
$\rt^2_2$ is incomparable with weak K\"onig's lemma.
Although the above mentioned results show that compactness is not really
necessary in the proof of~$\rt^2_2$, $\wkl$
preserves many computability-theoretic notions and is therefore
involved in many effective constructions related to~$\rt^2_2$.
Flood~\cite{Flood2012Reverse} introduced recently a Ramsey-type version of K\"onig's lemma ($\rwkl$).
This strict weakening of~$\wkl$ is sufficient in most applications of~$\wkl$
involved in proofs of~$\rt^2_2$. The statement~$\rwkl$ has been later studied by Bienvenu, Patey and Shafer~\cite{Bienvenu2015logical}
and by Flood and Towsner~\cite{Flood2014Separating}.

\begin{theorem}\label{thm:wkl-n-fairness}
For every $n \geq 1$, $\wkl$ admits $n$-fairness preservation.
\end{theorem}
\begin{proof}
Let~$C$ be a set $n$-fair for some sets~$A_0, A_1 \subseteq 2^{<\omega}$,
and let~$T \subseteq 2^{<\omega}$ be a $C$-computable infinite binary tree.
We construct an infinite decreasing sequence of computable subtrees $T = T_0 \supseteq T_1 \supseteq \dots$
such that for every path $P$ through $\bigcap_s T_s$, $P \oplus C$ is $n$-fair for $A_0, A_1$.
Note that the intersection~$\bigcap_s T_s$ is non-empty since the $T$'s are infinite trees.
More precisely, if we interpret $s$ as a tuple~$\tuple{m, \varphi, M}$ where $\varphi(G,U)$
is an $m$-by-$2^nm$ $\Sigma^{0,C}_1$ formula~$\varphi(G,U)$ and $M$ is an $m$-by-$2^nm$ disjoint matrix $M$,
we want to satisfy the following requirement.

\begin{quote}
$\Rcal_s$ : For every path~$P$ through~$T_{s+1}$, either $\varphi(P, U)$ is not essential in $M$,
or~$\varphi(P, V)$ holds for some $M$-valuation $V$ diagonalizing against $A_0, A_1$.
\end{quote}

Given two $M$-valuations $V_0 = (B_{i,j} : i < m, j < 2^nm)$ and $V_1 = (D_{i,j} :  i < m, j < 2^nm)$, 
we write $V_0 \subseteq V_1$
to denote the pointwise subset relation, that is, for every $i < m$ and every $j < 2^nm$, $B_{i,j} \subseteq D_{i,j}$.
At stage~$s = \tuple{m, \varphi, M}$, given some infinite, computable binary tree~$T_s$, 
define the $m$-by-$2^nm$ $\Sigma^{0,C}_1$ formula
\[
\psi(U) = (\exists n)(\forall \tau \in T_s \cap 2^n)(\exists \tilde{V} \subseteq U)\varphi(\tau, \tilde{V})
\]
We have two cases.
In the first case, $\psi(U)$ is not essential in $M$ with some witness~$t$. By compactness,
the following set is an infinite $C$-computable subtree of~$T_s$:
\[
T_{s+1} = \{ \tau \in T_s : (\mbox{for every } M\mbox{-valuation } V > t)\neg \varphi(\tau, V) \}
\]
The tree $T_{s+1}$ has been defined so that $\varphi(P, U)$
is not essential in $M$ for every~$P \in [T_{s+1}]$.

In the second case, $\psi(U)$ is essential in $M$. By $n$-fairness of $C$ for $A_0, A_1$,
there is an $M$-valuation $V$ diagonalizing against $A_0, A_1$ such that $\psi(V)$ holds.
We claim that for every path~$P \in [T_s]$,
$\varphi(P, \tilde{V})$ holds for some $M$-valuation~$\tilde{V}$ diagonalizing against~$A_0, A_1$.
Fix some path~$P \in [T_s]$. Unfolding the definition of~$\psi(V)$, there is some~$n$ such that $\varphi(P \uh n, \tilde{V})$ holds
for some~$M$-valuation~$\tilde{V} \subseteq V$. 
Since $V$ is diagonalizing against~$A_0, A_1$, for every $i < m$, there is some~$L, R \in V(i)$ 
such that $L \subseteq A_0$ and $R \subseteq A_1$.
Let $\tilde{L}, \tilde{R} \in \tilde{V}(i)$ be such that $\tilde{L} \subseteq L$ and $\tilde{R} \subseteq R$.
In particular, $\tilde{L} \subseteq A_0$ and $\tilde{R} \subseteq A_1$ so $\tilde{V}$ diagonalizes against $A_0, A_1$.
Take~$T_{s+1} = T_s$ and go to the next stage.
This finishes the proof of Theorem~\ref{thm:wkl-n-fairness}.
\end{proof}

As previously noted, preserving $n$-fairness for every~$n$ implies preserving fairness.
However, we really need the fact that $\wkl$ admits $n$-fairness preservation
and not only fairness preservation in the proof of Theorem~\ref{thm:rt12-strong-fairness}.

\begin{corollary}\label{cor:wkl-fairness-preservation}
$\wkl$ admits fairness preservation.
\end{corollary}

Cholak, Jockusch and Slaman~\cite{Cholak2001strength} studied
extensively Ramsey's theorem for pairs in reverse mathematics,
and introduced their cohesive and stable variants.

\begin{definition}[Cohesiveness]
An infinite set $C$ is $\vec{R}$-cohesive for a sequence of sets $R_0, R_1, \dots$
if for each $i \in \omega$, $C \subseteq^{*} R_i$ or $C \subseteq^{*} \overline{R_i}$.
$\coh$ is the statement ``Every uniform sequence of sets $\vec{R}$
has an $\vec{R}$-cohesive set.''
\end{definition}

A coloring $f : [\omega]^{k+1} \to n$ is \emph{stable} if for every~$k$-tuple~$\sigma \in [\omega]^k$,
$\lim_s f(\sigma, s)$ exists. $\srt^k_n$ is the restriction of~$\rt^k_n$ to stable colorings.
Mileti~\cite{Mileti2004Partition} and Jockusch \& Lempp [unpublished]
proved that~$\rt^2_2$ is equivalent to~$\srt^2_2+\coh$ over~$\rca$. Recently,
Chong et al.~\cite{Chong2014metamathematics} proved that~$\srt^2_2$ is strictly weaker than~$\rt^2_2$
over~$\rca$. However they used non-standard models to separate the statements and the question
whether~$\srt^2_2$ and~$\rt^2_2$ coincide over~$\omega$-models remains open.

Cohesiveness can be seen as a sequential version
of~$\rt^1_2$ with finite errors. There is a natural decomposition of~$\rt^2_2$
between $\coh$ and $\Delta^0_2$ instances of~$\rt^1_2$. 
Indeed, given a computable instance~$f : [\omega]^2 \to 2$ of~$\rt^2_2$,
$\coh$ states the existence of an infinite set~$H$ such that $f : [H]^2 \to 2$ is stable.
By Schoenfield's limit lemma~\cite{Shoenfield1959degrees},
the stable coloring $f : [H]^2 \to 2$ can be seen as the $\Delta^0_2$ approximation of a $\emptyset'$-computable
instance $\tilde{f} : H \to 2$ of~$\rt^1_2$. Moreover, we can $H$-compute an infinite $f$-homogeneous set 
from any $\tilde{f}$-homogeneous set.
We shall therefore prove independently fairness preservation of~$\coh$
and strong fairness preservation of~$\rt^1_2$ to deduce that~$\rt^2_2$ admits fairness preservation.

\begin{theorem}\label{thm:coh-n-fairness}
For every $n \geq 1$, $\coh$ admits $n$-fairness preservation.
\end{theorem}
\begin{proof}
Let~$C$ be a set $n$-fair for some sets~$A_0, A_1 \subseteq 2^{<\omega}$,
and let~$R_0, R_1, \dots$ be a $C$-computable sequence of sets.
We will construct an $\vec{R}$-cohesive set $G$ such that 
$G \oplus C$ is $n$-fair for~$A_0, A_1$.
The construction is done by a Mathias forcing, whose conditions are pairs $(F, X)$
where $X$ is a $C$-computable set. The result is a direct consequence of the following lemma.

\begin{lemma}\label{lem:coh-preservation-lemma}
For every condition~$(F, X)$, every $m$-by-$2^nm$ $\Sigma^{0,C}_1$ formula~$\varphi(G, U)$
and every $m$-by-$2^nm$ disjoint matrix $M$, there exists an extension~$d = (E, Y)$ such that
either $\varphi(G, U)$ is not essential for every set $G$ satisfying $d$, or $\varphi(E, V)$ holds
for some $M$-valuation $V$ diagonalizing against $A_0, A_1$.
\end{lemma}
\begin{proof}
Define the $m$-by-$2^nm$ $\Sigma^{0,C}_1$ formula
$\psi(U) = (\exists G \supseteq F)[(G \subseteq F \cup X) \wedge \varphi(G, U)]$.
By $n$-fairness of $C$ for $A_0, A_1$, 
either $\psi(U)$ is not essential in $M$, or~$\psi(V)$ holds for some $M$-valuation $V$
diagonalizing against $A_0, A_1$.
In the former case, the condition~$(F,X)$ already satisfies the desired property.
In the latter case, by the finite use property, there exists a finite set~$E$ satisfying~$(F, X)$ such that~$\varphi(E, V)$ holds.
Let $Y = X \setminus [0, max(E)]$. The condition $(E, Y)$ is a valid extension.
\end{proof}

Using Lemma~\ref{lem:coh-preservation-lemma}, define an infinite descending sequence 
of conditions~$c_0 = (\emptyset, \omega) \geq c_1 \geq \dots$
such that for each~$s \in \omega$
\begin{itemize}
	\item[(i)] $|F_s| \geq s$
	\item[(ii)] $X_{s+1} \subseteq R_s$ or $X_{s+1} \subseteq \overline{R}_s$
	\item[(iii)] $\varphi(G, U)$ is not essential in $M$ for every set $G$ satisfying $c_{s+1}$, 
	or $\varphi(F_{s+1}, V)$ holds
	for some $M$-valuation $V$ diagonalizing against $A_0, A_1$ if~$s = \tuple{\varphi, M}$
\end{itemize}
where~$c_s = (F_s, X_s)$. The set $G = \bigcup_s F_s$ is infinite by (i),
$\vec{R}$-cohesive by (ii) and $G \oplus C$ is $n$-fair for~$A_0, A_1$ by (iii).
This finishes the proof of Theorem~\ref{thm:coh-n-fairness}.
\end{proof}

\begin{corollary}
$\coh$ admits fairness preservation.
\end{corollary}

The next theorem is the reason why we use the notion of
fairness instead of $n$-fairness
in our separation of~$\rt^2_2$ from~$\tto^2_2$.
Indeed, given an instance of~$\rt^1_2$ and a set~$C$
which is $n$-fair for some sets~$A_0, A_1$, the proof
constructs a solution $H$ such that~$H \oplus C$ is $(n+1)$-fair for~$A_0, A_1$.
We shall see in Corollary~\ref{cor:rt12-not-strong-1-fairness} that the proof is optimal,
in the sense that $\rt^1_2$ does not admit strong~$n$-fairness preservation.

\begin{theorem}\label{thm:rt12-strong-fairness}
$\rt^1_2$ admits strong fairness preservation.
\end{theorem}
\begin{proof}
Let~$C$ be a set $n$-fair for some sets~$A_0, A_1 \subseteq 2^{<\omega}$,
and let~$B_0 \cup B_1 = \omega$ be a (non-necessarily computable) 2-partition of~$\omega$.
Suppose that there is no infinite set $H \subseteq B_0$ or $H \subseteq B_1$
such that $H \oplus C$ is $n$-fair for~$A_0, A_1$, since otherwise we are done.
We construct a set $G$ such that both $G \cap B_0$ and $G \cap B_1$ are infinite.
We need therefore to satisfy the following requirements for each~$p \in \omega$.
\[
  \Ncal_p : \hspace{20pt} (\exists q_0 > p)[q_0 \in G \cap B_0] 
		\hspace{20pt} \wedge \hspace{20pt} (\exists q_1 > p)[q_1 \in G \cap B_1] 
\]
Furthermore, we want to ensure that one of $(G \cap B_0) \oplus C$ 
and $(G \cap B_1) \oplus C$ is fair for~$A_0, A_1$. To do this, we will satisfy the following requirements
for every integer $m$, every $m$-by-$2^{n+1}m$ $\Sigma^{0,C}_1$ formulas $\varphi_0(H, U)$ and  $\varphi_1(H, U)$
and every $m$-by-$2^{n+1}m$ disjoint matrices $M_0$ and $M_1$.
\[
  \Qcal_{\varphi_0, M_0, \varphi_1, M_1} : \hspace{20pt} 
		\Rcal_{\varphi_0,M_0}^{G \cap B_0} \hspace{20pt} \vee \hspace{20pt}  \Rcal_{\varphi_1,M_1}^{G \cap B_1}
\]
where $\Rcal_{\varphi, M}^H$ holds if $\varphi(H, U)$ is not essential in $M$
or $\varphi(H, V)$ holds for some $M$-valuation $V$ diagonalizing against $A_0, A_1$.
We first justify that if every $\Qcal$-requirement is satisfied, then either $(G \cap B_0) \oplus C$
or $(G \cap B_1) \oplus C$ is $(n+1)$-fair for $A_0, A_1$.
By the usual pairing argument, for every~$m$, there is some side $i < 2$ such that
the following property holds:
\begin{quote}
(P) For every $m$-by-$2^{n+1}m$ $\Sigma^{0,C}_1$ formula $\varphi(G \cap B_i, U)$ 
and every $m$-by-$2^{n+1}m$ disjoint matrix $M$, either $\varphi(G \cap B_i, U)$
is not essential in $M$, or $\varphi(G \cap B_i, V)$ holds for some $M$-valuation~$V$
diagonalizing against $A_0, A_1$.
\end{quote}
By the infinite pigeonhole principle, 
there is some side $i < 2$ such that (P) holds for infinitely many~$m$.
By a cropping argument, if (P) holds for $m$ and $q < m$, then (P) holds for~$q$.
Therefore (P) holds for every $m$ on side~$i$. In other words, $(G \cap B_i) \oplus C$
is $(n+1)$-fair for~$A_0, A_1$.

We construct our set $G$ by forcing. Our conditions are Mathias conditions~$(F, X)$,
such that~$X \oplus C$ is $n$-fair for~$A_0, A_1$.
We now prove the progress lemma, stating that we can force both $G \cap B_0$
and $G \cap B_1$ to be infinite.

\begin{lemma}\label{lem:rt12-fairness-progress}
For every condition~$c = (F, X)$, every $i < 2$ and every~$p \in \omega$
there is some extension~$d = (E, Y)$ such that~$E \cap B_i \cap (p,+\infty) \neq \emptyset$.
\end{lemma}
\begin{proof}
Fix~$c$, $i$ and~$p$. If $X \cap B_i \cap (p,+\infty) = \emptyset$,
then $X \cap (p,+\infty)$ is an infinite subset of~$B_{1-i}$.
Moreover, $X \cap (p,+\infty)$ is $n$-fair for~$A_0, A_1$, contradicting our hypothesis.
Thus, there is some~$q > p$ such that $q \in X \cap B_i \cap (p,+\infty)$.
Take $d = (F \cup \{q\}, X \setminus [0,q])$ as the desired extension.
\end{proof}

Next, we prove the core lemma stating that we can satisfy each $\Qcal$-requirement.
A condition~$c$ \emph{forces} a requirement~$\Qcal$
if $\Qcal$ is holds for every set~$G$ satisfying~$c$.
This is the place where we really need the fact that~$\wkl$
admits $n$-fairness preservation and not only fairness preservation.

\begin{lemma}\label{lem:rt12-fairness-forcing}
For every condition~$c = (F, X)$, every integer $m$, every $m$-by-$2^{n+1}m$ $\Sigma^{0,C}_1$ formulas $\varphi_0(H, U)$ and  $\varphi_1(H, U)$ and every $m$-by-$2^{n+1}m$ disjoint matrices $M_0$ and $M_1$,
there is an extension~$d = (E, Y)$ forcing~$\Qcal_{\varphi_0, M_0, \varphi_1, M_1}$.
\end{lemma}
\begin{proof}
Let~$\psi(U_0,U_1)$ be the $2m$-by-$2^{n+1}m$ $\Sigma^{0,X \oplus C}_1$ formula which holds
if for every 2-partition~$Z_0 \cup Z_1 = X$, there is some~$i < 2$,
some finite set $E \subseteq Z_i$
and an $m$-by-$2^{n+1}m$ $M_i$-valuation $V \subseteq U_i$ such that $\varphi_i((F \cap B_i) \cup E, V)$
holds. By $n$-fairness of~$X \oplus C$, we have two cases.

In the first case, $\psi(U_0, U_1)$ is not essential in~$M_0,M_1$, with some witness~$t$.
By compactness, the $\Pi^{0, X \oplus C}_1$ class~$\Ccal$ of sets~$Z_0 \oplus Z_1$
such that $Z_0 \cup Z_1 = \omega$ and for every~$i < 2$
and every finite set $E \subseteq Z_i$, there is no $M_i$-valuation $V > t$
such that $\varphi_i((F \cap B_i) \cup E, V)$ holds is non-empty.
By $n$-fairness preservation of~$\wkl$ (Theorem~\ref{thm:wkl-n-fairness}), there is a 2-partition
$Z_0 \oplus Z_1 \in \Ccal$ such that~$Z_0 \oplus Z_1 \oplus C$ is $n$-fair for~$A_0, A_1$.
Since $Z_0 \cup Z_1 = X$, there is some~$i < 2$ such that~$Z_i$ is infinite. Take such an~$i$.
The condition~$d = (F, Z_i)$ is an extension forcing $\Qcal_{\varphi_0, M_0, \varphi_1, M_1}$
by the $i$th side.

In the second case, $\psi(V_0, V_1)$ holds for some $(M_0,M_1)$-valuation $(V_0, V_1)$
diagonalizing against $A_0, A_1$. Let~$Z_0 = X \cap B_0$ and~$Z_1 = X \cap B_1$.
By hypothesis, there is some $i < 2$, some finite set~$E \subseteq Z_i = X \cap B_i$
and some $M_i$-valuation~$V \subseteq V_i$ such that $\varphi_i((F \cap B_i) \cup E, V)$ holds.
Since $V \subseteq V_i$, the $M_i$-valuation~$V$ diagonalizes against~$A_0, A_1$.
The condition~$d = (F \cup E, X \setminus [0, max(E)])$ is an extension
forcing $\Qcal_{\varphi_0, M_0, \varphi_1, M_1}$ by the $i$th side.
\end{proof}

Using Lemma~\ref{lem:rt12-fairness-progress} and Lemma~\ref{lem:rt12-fairness-forcing}, define an infinite descending sequence 
of conditions~$c_0 = (\emptyset, \omega) \geq c_1 \geq \dots$
such that for each~$s \in \omega$
\begin{itemize}
	\item[(i)] $|F_s \cap B_0| \geq s$ and~$|F_s \cap B_1| \geq s$
	\item[(ii)] $c_{s+1}$ forces~$\Qcal_{\varphi_0, M_0, \varphi_1, M_1}$
	if $s = \tuple{\varphi_0, M_0, \varphi_1, M_1}$
\end{itemize}
where~$c_s = (F_s, X_s)$. Let~$G = \bigcup_s F_s$.
The sets~$G \cap B_0$ and~$G \cap B_1$ are both infinite by (i)
and one of~$G \cap B_0$ and~$G \cap B_1$ is fair for~$A_0, A_1$ by (ii).
This finishes the proof of Theorem~\ref{thm:rt12-strong-fairness}.
\end{proof}

\begin{theorem}\label{thm:rt22-fairness-preservation}
$\rt^2_2$ admits fairness preservation.
\end{theorem}
\begin{proof}
Fix any set~$C$ fair for some sets $A_0, A_1 \subseteq 2^{<\omega}$ and any $C$-computable
coloring $f : [\omega]^2 \to 2$.
Consider the uniformly~$C$-computable sequence of sets~$\vec{R}$ defined for each~$x \in \omega$ by
\[
R_x = \{s \in \omega : f(x,s) = 1\}
\]
As~$\coh$ admits fairness preservation, there is
some~$\vec{R}$-cohesive set~$G$ such that $G \oplus C$ is fair for $A_0, A_1$.
The set~$G$ induces a $(G \oplus C)'$-computable coloring~$\tilde{f} : \omega \to 2$ defined by:
\[
(\forall x \in \omega) \tilde{f}(x) = \lim_{s \in G} f(x,s)
\]
As~$\rt^1_2$ admits strong fairness preservation,
there is an infinite $\tilde{f}$-homogeneous set~$H$ such that
$H \oplus G \oplus C$ is fair for $A_0, A_1$.
The set $H \oplus G \oplus C$ computes an infinite $f$-homogeneous set.
\end{proof}

\begin{corollary}\label{cor:rt22+wkl-not-tt22}
$\rt^2_2 \wedge \wkl$ does not imply $\tto^2_2$ over~$\rca$.
\end{corollary}
\begin{proof}
By Theorem~\ref{thm:rt22-fairness-preservation}
and Corollary~\ref{cor:wkl-fairness-preservation}, $\rt^2_2$ and $\wkl$ admit fairness preservation.
By Theorem~\ref{thm:tt22-not-fairness}, $\tto^2_2$ does not admit fairness preservation.
We conclude by Lemma~\ref{lem:fairness-provability}.
\end{proof}

We now prove the optimality of Theorem~\ref{thm:rt12-strong-fairness} and Theorem~\ref{thm:rt22-fairness-preservation}
by showing that~$n$-fairness cannot be preserved.

\begin{theorem}\label{thm:rt22-not-1-fairness-preservation}
$\srt^2_2$ does not admit $n$-fairness preservation for any $n \geq 1$.
\end{theorem}
\begin{proof}
Let~$A_0 \cup A_1 = 2^{<\omega}$ be the $\Delta^0_2$ partition constructed in Lemma~\ref{lem:partition-emptyset-fair}.
By Schoenfield's limit lemma~\cite{Shoenfield1959degrees}, there is a stable computable function $f : [\omega]^2 \to 2$
such that $x \in A_{\lim_s f(x,s)}$ for every~$x$.
Fix some $n \geq 1$.
For each~$\sigma \in 2^{n+1}$, apply $\srt^2_2$ to the coloring~$f$ restricted to
the set $S_\sigma = \{\tau \succeq \sigma\}$
to obtain an infinite $f$-homogeneous set $H_\sigma$ for some color~$c_\sigma < 2$.
By definition of~$f$, $H_\sigma \subseteq A_{c_\sigma}$.
By the finite pigeonhole principle, there is a color~$c < 2$ and a set $M \subseteq 2^{n+1}$ of size $2^n$
such that $c_\sigma = c$ for every $\sigma \in M$. We can see $M$ as a 1-by-$2^n$ disjoint matrix.
Let $H = \bigoplus_{\sigma \in M} H_\sigma$ and let $\varphi(U_\sigma : \sigma \in M)$ be the $1$-by-$2^n$ $\Sigma^{0,H}_1$ formula
which holds if for every $\sigma \in M$, $U_\sigma$ is a non-empty subset of $H_\sigma$.
Note that $H_\sigma \subseteq A_c$ for every $\sigma \in M$.
The formula $\varphi(U)$ is essential in $M$ but there is no $M$-valuation $V = (V_\sigma : \sigma \in M)$ 
such that $\varphi(V)$ holds and $V_\sigma \subseteq A_{1-c}$ for some $\sigma \in M$.
Therefore $H$ is not $n$-fair for~$A_0, A_1$.
\end{proof}

\begin{corollary}\label{cor:rt12-not-strong-1-fairness}
$\rt^1_2$ does not admit $n$-fairness preservation for every~$n \geq 1$.
\end{corollary}
\begin{proof}
Fix some~$n \geq 2$. By Theorem~\ref{thm:rt22-not-1-fairness-preservation},
there is some set $C$ $n$-fair some $A_0, A_1$ and a stable $C$-computable
function $f : [\omega]^2 \to 2$ such that for every infinite $f$-homogeneous set~$H$,
$H \oplus C$ is not $n$-fair for~$A_0, A_1$.
Let~$\tilde{f} : \omega \to 2$ be defined by $\tilde{f}(x) = \lim_s f(x,s)$.
Every infinite $\tilde{f}$-homogeneous set $H$ $C$-computes an infinite 
$f$-homogeneous set $H_1$ such that $H_1 \oplus C$ is not $n$-fair for~$A_0, A_1$.
Therefore $H \oplus C$ is not $n$-fair for~$A_0, A_1$.
\end{proof}

%% file: parts/part4-questions.tex
In this last section, we state some remaining open questions.
The tree theorem for pairs is known to lie between~$\aca$ and~$\rt^2_2$ over~$\rca$.
By Corollary~\ref{cor:rt22+wkl-not-tt22}, $\tto^2_2$ is strictly stronger
than~$\rt^2_2$ over~$\rca$.
However, it is unknown whether~$\tto^2_2$ is strictly weaker than~$\aca$ over~$\rca$.

\begin{question}
Does~$\tto^2_2$ imply~$\aca$ over~$\rca$?
\end{question}

From a computability-theoretic point of view, Seetapun~\cite{Seetapun1995strength} proved
that for every non-computable set~$C$, every computable instance of~$\rt^2_2$
has a solution which does not compute~$C$. This is the notion of \emph{cone avoidance}.

\begin{question}
Does~$\tto^2_2$ admit cone avoidance?
\end{question}

Dzhafarov and Jockusch~\cite{Dzhafarov2009Ramseys} simplified
Seetapun's argument and proved that for every non-computable set~$C$,
every arbitrary, that is, non-necessarily computable, instance of~$\rt^1_2$ has a solution
which does not compute~$C$. This strengthening is called~\emph{strong cone avoidance}
and is usually joined with the cone avoidance of the cohesive version of the principle
to obtain cone avoidance for the principle over pairs.

\begin{question}
Does~$\tto^1_2$ admit strong cone avoidance?
\end{question}